\documentclass[12pt]{amsart}
\usepackage{here}
\usepackage{tikz}

\usepackage{amsmath, amssymb}
\usepackage{pifont}
\usepackage{booktabs}

\numberwithin{equation}{section}
\newtheorem{theorem}{Theorem}[section]  
\newtheorem{theorem?}{``Theorem''}[section]  
\newtheorem{corollary}[theorem]{Corollary}
\newtheorem{proposition}[theorem]{Proposition}
\newtheorem{lemma}[theorem]{Lemma}

\theoremstyle{definition}

\theoremstyle{remark}
\newtheorem{remark}[theorem]{Remark}  

\newcommand{\R}{{\mathbb R}}
\newcommand{\C}{{\mathbb C}}

\newcommand{\N}{{\mathbb N}}
\newcommand{\Z}{{\mathbb Z}}

\renewcommand{\a}{\alpha}

\begin{document}
\title[Oscillatory intagrals]
{
Asymptotic expansion of 
oscillatory integrals \\ 
with singular phases 
%Asymptotic expansion of oscillatory integrals whose phase has a singularity
} 
\author{Joe Kamimoto and Hiromichi Mizuno}
%\dedicatory{To the memory of Professor Mihnea Col\c{t}oiu.}
\address{Faculty of Mathematics, Kyushu University, 
Motooka 744, Nishi-ku, Fukuoka, 819-0395, Japan} 
\email{
joe@math.kyushu-u.ac.jp}
\address{Faculty of Mathematics, Kyushu University, 
Motooka 744, Nishi-ku, Fukuoka, 819-0395, Japan}
\email{mizuno.hiromichi.965@s.kyushu-u.ac.jp}
%%%
\maketitle
%\date{}

%\vspace{9 em}

\begin{abstract}
The purpose of this article is to 
describe the singularities of 
one-dimensional oscillatory integrals, 
whose phases have a certain singularity, 
in the form of an asymptotic expansion. 
In the case of the Laplace integral,  
an analogous result is also given. 

\end{abstract}

%\clearpage

%\tableofcontents

%%%%%%%%%%%%%%%%%%%%%%%
%\setcounter{section}{-1}

\section{Introduction}

In the investigation of various issues in mathematics, 
oscillatory integrals of the form 
\begin{equation}\label{eqn:1.1}
	I(t)=
\int_{-\infty}^{\infty} e^{ i t f(x)} \varphi(x)dx
	\quad\,\,\, \text{for $t \in \R$}
\end{equation}
appear and 
their properties often play important roles in it. 
Here $f$ and $\varphi$ are real-valued $C^{\infty}$ functions 
defined on an open interval in $\R$ containing the origin 
and the support of $\varphi$ is contained in this interval. 
The functions $f$ and $\varphi$ are called 
the {\it phase} and
the {\it amplitude} respectively. 

It is easy to see that 
the integral in (\ref{eqn:1.1}) can be regarded as 
a $C^{\infty}$ function of $t$ on $\R$.  
Moreover, 
the behavior of this function for large $t$ 
is well understood 
(cf. \cite{AGV88}, \cite{Hor90}, \cite{Ste93}).   
%In the case where the phase has a critical point,  
%its behavior can be clearly expressed 
%in the form of asymptotic expansion.
If there is $k\in\N$ with $k\geq 2$ such that 
$f'(0)=\cdots=f^{(k-1)}(0)=0$ and $f^{(k)}(0)\neq 0$,  
then for any positive integer $N$, 
\begin{equation}\label{eqn:1.2}
I(t)=\sum_{n=1}^{N-1} c_n t^{-n/k} + O(t^{-N/k})
\quad \mbox{as $t \to +\infty$},
\end{equation}
where $c_n$ are constants depending 
on $f$ and $\varphi$, and the exact value of 
the first coefficient $c_1$ can be explicitly given.
When $f$ does not have a critical point,
its behavior is obvious in the sense:
$I(t)=O(t^{-N})$ as $t\to\infty$ for 
any positive integer $N$.
More generally, 
in the case where $k$ is a positive real number,
the asymptotic expansion of $I(t)$ can be
expressed in a similar fashion to (\ref{eqn:1.2})
(see \cite{Nag20}).

%%%%%%%%%%%%%%%%%%%%%%%%%%
Corresponding to the above regular phase case, 
the case where $f$ has singularities
seems to be not well thought. 
In this article, we investigate
the properties of integrals of the form
\begin{equation}\label{eqn:1.3}
I_{\alpha}(t)=
\int_{0}^{\infty}\exp
\left( i t\frac{1}{x^{\alpha}}
\right) \varphi(x)dx
\quad\,\,\, \text{for $t \in \R$,}
\end{equation}
where $\alpha$ is a positive real number and 
$\varphi$ is the same as that in (\ref{eqn:1.1}). 
%%%
There have been many interesting numerical analyses 
of the integral (\ref{eqn:1.3})
(\cite{Has091}, \cite{Has092}, 
\cite{KaX10}, \cite{MoX11}, etc.), while
we will investigate this integral 
from the viewpoint of the asymptotic expansion.

From the convergence of the integral, 
the integral (\ref{eqn:1.3}) can be regarded as a function of $t$ 
defined on $\R$, but 
the smoothness of this function is not obvious. 
The convergence of the derivatives of 
the integrand in (\ref{eqn:1.3}) 
only implies that 
$I_{\alpha}(t)$ is a $C^{\lceil1/\alpha \rceil-1}$ function 
on $\R$. 
For example, 
$I_1(t)$ is continuous but 
it is not differentiable if $\varphi(0)>0$ 
and $\varphi\geq 0$ on $\R$.
It is interesting to consider 
what kind of singularities $I_{\alpha}(t)$ has. 
The following theorem exactly shows 
the singular part of $I_{\alpha}(t)$.

%%%%%%%%%%%%%%%%%%%%%%
\begin{theorem}\label{thm:1.1}
%Let $\alpha$ be a positive real number.
(i)\quad 
If $\alpha>0$ is a rational number, 
then for any positive integer $N$, %with $N>1/\alpha$, 
\begin{equation}\label{eqn:1.4}
I_{\alpha}(t) = 
\sum_{n=1}^{N} A_{n} t^{n/\alpha}+
\sum_{n=1}^{N} B_{n} t^{n/\alpha} \log t+\psi_N(t)
\quad \mbox{ for $t\in\R_+$ },
\end{equation}
where $\psi_N(t)$ is a 
$C^{\lceil \frac{N+1}{\alpha}\rceil-1}$ function on $\R_+$ and   
\begin{equation}\label{eqn:1.5}
\begin{split}
&A_n=\frac{e^{-\frac{n}{2\alpha}\pi  i }}{\alpha}
\frac{\varphi^{(n-1)}(0)}{(n-1)!}
\Gamma(-n/\alpha) \,\,\,
\mbox{if $n/\alpha \not\in\N$}, \quad 
A_n=0 \,\,\, \mbox{if $n/\alpha\in\N$,}\\
&B_n=\frac{-1}{\alpha}
\frac{\varphi^{(n-1)}(0)}{(n-1)!}
\frac{ i^{n/\alpha}}{(n/\alpha)!}
   \,\,\, 
\mbox{if $n/\alpha \in\N$},   \quad 
B_n=0 \,\,\, \mbox{if $n/\alpha\not\in\N$,}
\end{split}
\end{equation}
for $n\in\N$,
where $\Gamma$ means the Gamma function.

(ii)\quad 
If $\alpha>0$ is not a rational number, then
for any positive integer $N$, %with $N>1/\alpha$, 
\begin{equation}\label{eqn:1.6}
I_{\alpha}(t) = 
\sum_{n=1}^{N} A_{n} t^{n/\a}+\phi_N(t)
\quad \mbox{ for $t\in\R_+$,}
\end{equation}
where $\phi_N(t)$ is a $C^{\lceil \frac{N+1}{\alpha}\rceil-1}$ 
function on $\R$ and 
$A_n$ are as in (\ref{eqn:1.5})
\end{theorem}
%%%%%%%%%%%%%%%%%%%%%%%%%%%%

\begin{remark}
(1) \quad
If the branch of $t^{n/\alpha}$ for $t<0$ is chosen
as $t^{n/\alpha}=e^{ i n\pi/\alpha} |t|^{n/\alpha}$, then
the equation in (\ref{eqn:1.6}) holds for all $t\in\R$, 
where  $\phi_N(t)$ is a 
$C^{\lceil \frac{N+1}{\alpha}\rceil-1}$ 
function on $\R$. 

(2) \quad
We can obtain a similar result to (\ref{eqn:1.4}) 
in the case where $t\leq 0$.
If $\alpha>0$ is a rational number, 
then for any positive integer $N$, 
\begin{equation*}\label{eqn:1.}
I_{\alpha}(t) = 
\sum_{n=1}^{N} A_{n} t^{n/\alpha}+
\sum_{n=1}^{N} {B}_{n} t^{n/\alpha} \log |t|+\phi_N(t)
\quad \mbox{ for $t\in\R_-$, }
\end{equation*}
where $B_n$ are as in (\ref{eqn:1.5})  and 
$\phi_N(t)$ is a 
$C^{\lceil \frac{N+1}{\alpha}\rceil-1}$ function on $\R_-$. 

(3)\quad 
When $\alpha$ is rational, 
$\alpha$ is uniquely denoted by $\alpha=p/q$, 
where $p,q\in\N$ with $(p,q)=1$. 
In this case, 
the second term in (\ref{eqn:1.4}) can be rewritten as 
\begin{equation}\label{eqn:1.7}
\sum_{n=1}^{N} B_{n} t^n \log t = 
\sum_{n=1}^{\lfloor\frac{N}{q}\rfloor}
\check{B}_n t^{qn} \log t \quad 
\mbox{ for $t\in\R_+$,} 
\end{equation}
with
$$
\check{B}_n=
\frac{-q}{p}
\frac{\varphi^{(pn-1)}(0)}{(pn-1)!}
\frac{ i ^{qn}}{(qn)!} \quad \mbox{for $n\in\N$}.
$$
\end{remark}

It follows from the above theorem that 
$I_{\alpha}$ is smooth away from the origin.

%%%%%%%%%%%%%%%
\begin{corollary}
$I_{\alpha}(t)$ is a $C^{\infty}$ function on $\R\setminus\{0\}$.
\end{corollary}
%%%%%%%%%%%%%%%

The above property seems to be not so obvious, 
because large order derivatives of the integrand 
in (\ref{eqn:1.3}) with respect to $t$ are not
absolutely integrable. 

Let us focus on the behavior of $I_{\alpha}(t)$ near the origin. 
From the above theorem, 
we can give an asymptotic expansion of 
$I_{\alpha}(t)$ at the origin 
in the case where $\alpha$ is rational:
%%%%
\begin{equation}\label{eqn:1.7}
I_{\alpha}(t) = 
\sum_{n=1}^{N} A_{n} t^{n/\alpha}+
\sum_{n=1}^{N} B_{n} t^n \log t+
\sum_{n=0}^{\lceil \frac{N+1}{\alpha}\rceil-2} 
C_{n} t^n+ 
O(t^{\lceil \frac{N+1}{\alpha}\rceil-1})
\quad \mbox{as $t\to +0$},
\end{equation}
for any positive integer $N$ with $N\geq \alpha$, 
where $A_n, B_n$ are the same as those  in (\ref{eqn:1.5})
and $C_n$ are constants depending on 
$\varphi, \alpha$ and, 
in particular,  
\begin{equation}\label{eqn:1.8}
C_n= i^n 
\int_0^{\infty}\frac{\varphi(x)}{x^{\alpha n}}dx \quad
\mbox{ for $n=0,\ldots, \lceil \frac{1}{\alpha} \rceil -1$,}
\end{equation}
which can be directly obtained by the Taylor formula. 
When $\varphi(0)>0$ and $\varphi\geq 0$ on $\R$, 
the improper integral in (\ref{eqn:1.8}) is convergent
if and only if $n\in\Z_+$ satisfies 
$n \leq \lceil \frac{1}{\alpha} \rceil -1$.
When $\alpha$ is not rational,  $I_{\alpha}(t)$
admits the following asymptotic expansion at the origin.
\begin{equation}\label{eqn:1.7}
I_{\alpha}(t) = 
\sum_{n=1}^{N} A_{n} t^{n/\alpha}+
\sum_{n=0}^{\lceil \frac{N+1}{\alpha}\rceil-2} 
C_{n} t^n+ 
O(t^{\lceil \frac{N+1}{\alpha}\rceil-1})\quad 
\mbox{as $t\to +0$},
\end{equation}
for any positive integer $N$ with $N\geq \alpha$, 
where $A_n$, $C_n$ are as in 
(\ref{eqn:1.5}), (\ref{eqn:1.8}).
%%%%%%%%%%%%%%%%%%%%%%%%%%%%%%

In particular, the limit of $I_{\alpha}(t)$ as $t\to+0$ 
is $C_0=\int_0^{\infty}\varphi(x)dx$ as in (\ref{eqn:1.8}). 
In more detail, we can explicitly see  
the decay as $t\to +0$ of the function 
$\tilde{I}_{\alpha}(t):=I_{\alpha}(t)-C_0$,  
which is classified as in the three cases as follows. 

\begin{corollary}\label{cor:1.3}
%The leading term of the above asymptotic expansion of 
%$\tilde{I}(t)$ in Theorem 2.1 is as follows. 
\begin{enumerate}
\item If $0 < \alpha < 1$, then
$$
\lim_{t \to 0} \frac{\tilde{I}_{\alpha}(t)}{t} 
=  i  \int_0^{\infty} \frac{\varphi(x)}{x^\a}dx.
$$
%%%%%%%%%%%%
\item If $\alpha=1$, then
$$
\lim_{t \to +0} \frac{\tilde{I}_{\alpha}(t)}{t \log t} 
= - i  \varphi(0). 
$$
%%%%%%%%%%%%%%	
\item If $1 < \alpha$, then
$$
\lim_{t \to +0}  \frac{\tilde{I}_{\alpha}(t)}{t^{1/\a}} 
=\frac{1}{\a} \varphi(0) e^{-\frac{1}{2\a}\pi  i }
\Gamma \left(-\frac{1}{\a} \right).
$$
\end{enumerate}
\end{corollary}

\begin{remark}
The behavior of $I_{\alpha}(t)$ for large $t$ 
is obvious in the sense:
%%%%
\begin{equation}\label{eqn:1.10}
I_{\alpha}(t)=O(t^{-N})\quad
\mbox{ as $t\to\infty$,} 
\end{equation}
for any positive integer $N$. 
This can be easily shown as follows. 
By changing the integral variable
 as $y=x^{-\alpha}$, $I_{\alpha}(t)$
can be expressed as
\begin{equation*}%\label{eqn:}
I_{\alpha}(t)
=\frac{1}{\alpha} \int_0^{\infty}
e^{ i ty}\varphi(y^{-1/\alpha})y^{-1/\alpha-1}dy.
\end{equation*}
Repeated integrations of parts imply
\begin{equation*}
I_{\alpha}(t)
=\left(\frac{-1}{ i  t}\right)^N
\int_0^{\infty}
e^{ i ty}\varphi_N(y^{-1/\alpha})y^{-1/\alpha-N-1}dy,
\end{equation*} 
for any positive integer $N$, where
$\varphi_N$ is a $C^{\infty}$ function on $\R$
whose support is contained in that of $\varphi$.
The equation (\ref{eqn:1.10}) 
follows from the above equation. 
\end{remark}

\begin{remark}
Let us consider integrals of the form:
\begin{equation}\label{eqn:1.11}
	\hat{I}(t)=
\int_{0}^{\infty} e^{ i t /f(x)} \varphi(x)dx
	\quad\,\,\, \text{for $t \in \R$,}
\end{equation}
where $f,\varphi$ are as in (\ref{eqn:1.1}) and, 
moreover, $f$ satisfies 
$f(0)=f'(0)=\cdots=f^{(k-1)}(0)=0$ and 
$f^{(k)}(0)\neq 0$ with $k\geq 1$.  
If the support of $\varphi$ is sufficiently small, then
$\hat{I}(t)$ essentially satisfies the property (i) in Theorem~1.1
with $\alpha=k$. 
(The coefficients $A_j, B_j$ are slightly different from those 
in (\ref{eqn:1.5}).)
This can be easily shown 
by using the implicit function theorem.
\end{remark}

\medskip

{\it Notation and symbols.}

\begin{itemize}
\item 
We denote 
$
\R_{\pm}:=\{x\in\R:\pm x\geq 0\}
$ and 
$
\Z_{+}:=\{n\in\Z: n \geq 0\}.
$
\item
When $n \leq a < n+1$ where $n\in\Z$, 
set $\lfloor a \rfloor =n$.
When $n-1< a\leq n$ where $n\in\Z$, 
set $\lceil a \rceil =n$.
\item
For a  $C^{n}$ function $f$, 
the $k$-th derivative of $f$ is denoted by 
$f^{(k)}$ for $k=1,\ldots, n$.
\item
Let $f(t),g(t)$ be functions defined 
on an interval $I\subset\R$.
We write 
$f(t)\equiv g(t)$ mod $C^{\infty}(I)$ to
express that $f(t)-g(t)$ is a $C^{\infty}$ function on $I$.
\end{itemize}

%%%%%%%%%%%%%%%%%%%%%%%%%%%%%%%%%%%%%%%%%%%%%%%

%%%%%%%%%%%%%%%%%%%%%%%%%%%%%%%%%%%%%%%%%%%%%%%%
\section{Some Fourier Transforms}

To prove Theorem~1.1,  
we prepare some auxiliary lemmas concerning 
the Fourier transform of some functions.
Let $p$ be a positive real number and
let $\chi:\R\to\R$ be a $C^{\infty}$ function satisfying that
\begin{eqnarray}\label{eqn:2.1}
		\chi(x) = \left \{ \begin{array}{ll}
		1	& \mbox{if $x \ge M$,}\\
		0      & \mbox{if $x \le L$,}\\
		\end{array} \right.
\end{eqnarray}
where $L,M$ are positive constants with $L < M$. 
For $t\in\R$, let $F_p(t)$ be the integral defined by
\begin{equation}\label{eqn:2.2}
	F_p(t)=\int_0^{\infty}e^{ i tx}x^{-p-1}\chi(x)dx.
\end{equation}
Note that this integral is the Fourier transform of 
the function $x\mapsto x^{-p-1}\chi(x)$ and that 
the integral in (\ref{eqn:2.2}) absolutely converges. 

%%%%%%%%%%%%%%%%%%%%%%%%%%%%%%%%%%%%%%%%

The following lemma plays crucial roles in the computation below.

%%%%%%%%%%%%%%%%%%%%%%%%%%%%%%%%%%%%%%%%%%%%%
\begin{lemma}\label{lem:2.1}
If $0 < \alpha < 1$, then 
\begin{equation}\label{eqn:2.3}
\int_0^{\infty} e^{ i y} y^{\alpha-1}dy
=
e^{\frac{\alpha}{2}\pi  i } \Gamma(\alpha). 
\end{equation}
\end{lemma}
%%%%%%%%%%%%%%%%%%%%%%%%%%%%%%%%%%%%%%
The above integral in the lemma may be considered 
as a generalized Fresnel integral and its value can be 
explicitly computed by using an elementary method 
of complex analysis. 
Indeed, its complete proof has been given in 
\cite{Nag20}, \cite{NaM20}, \cite{NaM22}, 
\cite{Tak21}, etc.  
Since the proof itself is essential to the analysis 
of this paper, 
we will recall its proof in Section~4. 
%%%%%%%%%%%%%%%%%%%%%%%%%%%

\medskip

Let us consider the property of $F_p(t)$ near the origin,
which depends on whether $p$ is an integer or not.
%so we consider each case at Lemma \ref{lem2} and \ref{lem3}.	

%%%%%%%%%%%%%%%%%%%%%%%%%%%%%%%%
\begin{lemma}\label{lem:2.2}
If $p>-1$ is not an integer, 
then 
%\begin{equation}	
$F_p(t)\equiv \tilde{A}_p t^{p}$ mod $C^{\infty}(\R)$,
%\end{equation} 
where 
%\begin{equation}
$\tilde{A}_p = e^{\frac{-p}{2}\pi  i }\Gamma(-p).$ 
%\end{equation}
Here, the branch of $t^p$ for $t<0$ is chosen as 
$t^p=e^{\pi p  i }|t|^p$.
\end{lemma}
%%%%%%%%%%%%%%%%%%%%%%%%%%%%%%%%

\begin{proof}
First, we consider the case where $p\in(-1,0)$. 
By using $\chi(x)$ in (\ref{eqn:2.1}), 
$F_p(t)$ is divided into two parts as follow.
\begin{eqnarray}
%\begin{split}
&&F_p(t)=\int_0^{\infty} e^{ i tx} x^{-p-1} \chi(x) dx 
\nonumber \\
&&\quad\quad = \int_0^{\infty} e^{ i tx} x^{-p-1} dx
- \int_0^{\infty} e^{ i tx} x^{-p-1} (1-\chi(x)) dx
\quad \mbox{ for $t\in\R$.}
\label{eqn:2.4}
%\end{split}
\end{eqnarray}
Since the support of $(1-\chi(x))$ is compact,
the second integral in (\ref{eqn:2.4}) is a $C^{\infty}$ function 
of $t$ on $\R$. 
When $t>0$,
by the change of the integral variable, 
the first integral in (\ref{eqn:2.4}) can be written as
\begin{equation}\label{eqn:2.5}
\int_0^{\infty} e^{ i tx} x^{-p-1} dx
= t^{p} \int_0^{\infty} e^{ i y} y^{-p-1} dy 
= t^{p} e^{\frac{-p}{2}\pi  i }\Gamma(-p)%(=: A_p t^{p})
\end{equation}
by using Lemma~2.1. 
When $t<0$, by choosing the branch of $t^p$ as in the lemma, 
\begin{equation}\label{eqn:}
\int_0^{\infty} e^{ i tx} x^{-p-1} dx
= |t|^{p} \int_0^{\infty} e^{- i y} y^{-p-1} dy 
= |t|^{p} e^{\frac{p}{2}\pi  i }\Gamma(-p)
= t^{p} e^{\frac{-p}{2}\pi  i }\Gamma(-p). %(=: A_p t^{p})
\end{equation}
Therefore, the equation can be obtained 
in the lemma for $p\in(-1,0)$. 

%%%%%%%%%%%%%%%%%%%%%%%%%%%%%%%%%%%%%

Next, let us consider the case where 
$p\in\R_+\setminus\Z$ with $p>1$.
Since $\chi(x)=0$ for $x \le L$,
an integration by parts implies 
%%%
\begin{eqnarray}
&&
F_p(t)=\int_0^{\infty} e^{ i tx} x^{-p-1}\chi(x)dx 
\nonumber\\
&&\quad \,\,=\frac{-1}{p}
\left[ x^{-p}e^{ i tx}\chi(x) \right]_0^{\infty}
+\frac{1}{p}
\int_0^{\infty}x^{-p}(e^{ i  tx}\chi(x))'dx
\nonumber\\
&&\quad \,\,= \frac{ i t}{p}
\int_0^{\infty}e^{ i tx}x^{-p}\chi(x)dx+
\frac{1}{p}
\int_0^{\infty}e^{ i tx}x^{-p}\chi'(x)dx
\quad \mbox{for $t\in\R$.}
\label{eqn:2.6}
\end{eqnarray}
%where $G_m(t)$ is a smooth function near the origin. 
	%\[
	%	G_n(t) 
	%	= \frac{(it)^{n-1}}{(p-1)\cdots(p-n)}\int_L^Me^{itc}c^{-p+n}\chi'(c)dc
	%\]
Note that the second integral in (\ref{eqn:2.6}) 
is a $C^{\infty}$ function of $t$ on $\R$.
Since there uniquely exist $m \in \N$ and $q\in (-1,0)$ 
such that $p = q + m$, the repeated process as the above implies
	\begin{equation}\label{eqn:2.7}
\begin{split}
F_p(t)\equiv \frac{( i t)^m}{p(p-1)\cdots (q+1)}
\int_0^{\infty}e^{ i tx}x^{-q-1}\chi(x)dx
\quad \mbox{mod $C^{\infty}(\R)$.}
\end{split}
\end{equation}
Since $q\in(-1,0)$, we can apply (\ref{eqn:2.5}) 
to (\ref{eqn:2.7}) and, 
as a result, we see that
$F_p(t)\equiv \tilde{A}_p t^{p}$ mod $C^{\infty}(\R)$, where
$$
\tilde{A}_p = 
\frac{ i ^m}{p(p-1)\cdots (q+1)}\tilde{A}_{q} %e^{\frac{-q+1}{2}\pi i}
=
\frac{(- i )^m}{(-p)(-p+1)\cdots(-q-1)}
e^{-\frac{q}{2}\pi  i } \Gamma(-q). 
$$
Using the equality $\Gamma(\alpha+1)=\alpha\Gamma(\alpha)$
for $\alpha\in\R\setminus\Z$, 
we can obtain the equation in the lemma
for $p\in\R_+\setminus\Z$.
\end{proof}
%%%%%%%%%%%%%%%%%%%%%%%%%%%%%%%%%%%

On the other hand, when $p$ is an integer, 
a logarithmic function appears in the singularity
of $F_p(t)$.

%%%%%%%%%%%%%%%%%%%%%%%%%%%%%
\begin{lemma}\label{lemma:2.3}
If $p$ is a nonnegative integer,	
then 
%\begin{equation}
$F_p(t)\equiv \tilde{B}_p t^{p}\log t$
%\end{equation}
mod $C^{\infty}(\R_+)$, where
$
\tilde{B}_p = - i^{p}/p!. 
$
(In particular, 
$F_0(t)\equiv -\log t$ mod 
$C^{\infty}(\R_+)$.)
\end{lemma}
%%%%%%%%%%%%%%%%%%%%%%%%%%%%%%%%%

\begin{proof}
It is sufficient to show the lemma in the case of $p=0$. 
Indeed, 
we can easily deal with the general case
in a similar argument to that in Lemma~2.2.

The integral $F_0(t)$ can be expressed as follows. 
\begin{equation*}
\begin{split}
F_0(t) 
&= \int_L^{1/t} e^{ i  t x}x^{-1}\chi(x)dx
+ \int_{1/t}^{\infty} e^{ i  t x}x^{-1}
\chi(x)dx \\
&=\int_L^{1/t} x^{-1}\chi(x)dx
+ \int_L^{1/t} (e^{ i  t x}-1)x^{-1}\chi(x)dx
+ \int_{1/t}^{\infty} e^{ i  t x}x^{-1}
\chi(x)dx \\
&=: G(t) + H(t) +K(t)
\quad \mbox{ for $t\in\R_+$.}
\end{split}
\end{equation*}

First, let us consider the integral $G(t)$. 
By integration by parts,
	\begin{equation*}
	\begin{split}
		G(t) 
		&= -\log t\cdot \chi(1/t)
		- \int_L^{1/t} \log x \cdot \chi'(x)dx \\
		&\equiv -\log t  \quad 
		\mbox{ mod $C^{\infty}(\R_+)$}	
	\end{split}
	\end{equation*}
%We remark that the support of $\chi'(x)$ is contained in $[L,M]$, 
%which implies the convergence of the last integral.

%%%%%%%%%%%%%%%%%%%%%%%%%%%%%%%%%%%%%%%%%%%%%%%
Second, let us show that $H(t)$ is a $C^{\infty}$ function 
on $\R$.
Let 
$$D:=\{(x,t)\in\R^2: tx<1,\,\, x>0, \,\, t\geq 0 \}.$$
The two $C^{\infty}$ functions 
$g,h:D \to \C$ can be defined by the convergent series
as follows:
%%%%%
\begin{equation*}%\label{eqn:}
g(x,t) 
%= \frac{e^{itx}-1}{x} 
= \sum_{n=1}^{\infty}\frac{ i^n}{n!}t^n x^{n-1},\quad
h(x,t) 
= \sum_{n=1}^{\infty}\frac{ i^n}{nn!}t^n x^n.
\end{equation*}
%%%%%%%%%%%%%%%%%%%
Noticing that $g(x,t)=(e^{ i  xt} -1)x^{-1}$ and 
$\frac{\partial}{\partial x}h(x,t)=g(x,t)$ on $D$, 
we have
\begin{equation}\label{eqn:2.8}
H(t)=\int_L^{\infty} g(x,t)\chi(x)dx
= h(1/t,t)
- \int_L^M h(x,t)\chi'(x)dx
\quad \mbox{ for $t\in\R_+$},
\end{equation}
by integration by parts.
It is easy to see that
$h(1/t,t)$ is a constant and that 
the last integral in (\ref{eqn:2.8}) is 
a $C^{\infty}$ function of $t$ on $\R_+$. 

%%%%%%%%%%%%%%%%%%%%%%%%%
Third, let us consider the integral $K(t)$. 
By changing the integral variable,
we have
\begin{eqnarray}
&&K(t)
=\int_{1/t}^{\infty} e^{ i  tx}x^{-1}\chi(x)dx
=\int_{1}^{\infty}e^{ i  y}y^{-1}\chi(y/t)dy 
\nonumber\\
&&\quad\quad=\int_{1}^{\infty}e^{ i  y}y^{-1}dy-
\int_{1}^{\infty}e^{ i  y}y^{-1}(1-\chi(y/t))dy
\quad \mbox{ for $t\in\R_+$.}
\label{eqn:2.9}
\end{eqnarray}
The first integral in (\ref{eqn:2.9}) is a constant 
defined by a convergent improper integral.
It is easy to see that the second integral in (\ref{eqn:2.9}) is 
a $C^{\infty}$ function of $t$ on $\R_+$.  
%%%%%%%%%%%%%%%%%%%%%%%%%%

Putting together the above results, 
we can see that 
$
F_0(t)+\log t
$
is a $C^{\infty}$ function on $\R_+$.
%%%%%%%%

\begin{remark}
For $t\in\R_-$, the same result in Lemma~2.3
can be obtained. 
Indeed, 
if $p$ is a nonnegative integer,	
then 
%\begin{equation}
$F_p(t)\equiv \tilde{B}_p t^{p}\log |t|$
%\end{equation}
mod $C^{\infty}(\R_-)$, where
$\tilde{B}_p$ is as in the lemma.
(In particular, 
$F_0(t)\equiv -\log |t|$ mod 
$C^{\infty}(\R_-)$.)
However, we do not know whether 
$F_p(t)-\tilde{B}_p t^{p}\log |t|$
is a $C^{\infty}$ function on $\R$
or not. 
\end{remark}

\end{proof}

%%%%%%%%%%%%%%%%%%%%%%%%%%

\section{Proof of Theorem~1.1}

Let $R$ be a positive number such that 
the support of $\varphi$ is contained in $[-R, R]$.

By changing the integral variable, we have 
$$
I_{\alpha}(t)
= \frac{1}{\a}
\int_0^{\infty} 
e^{ i  ty}\varphi(y^{-1/\a}) y^{-1/\a-1}dy
\quad \mbox{ for $t\in\R$.}
$$
%where $r = 1/R^{\alpha} ( > 0)$. 
By using $\chi$ in (\ref{eqn:2.1}) with $1/R^{\alpha}<L$, 
we divide the above integral as follows.
\begin{equation*}
I_{\alpha}(t)= J(t) + E(t),
\end{equation*}
with
\begin{eqnarray}
&&J(t) 
= \frac{1}{\a} \int_0^{\infty} 
e^{ i  ty}\varphi(y^{-1/\a}) y^{-1/\a-1}\chi(y) dy,
\label{eqn:3.1}\\
&&E(t) 
= \frac{1}{\a} \int_0^{\infty} 
e^{ i  ty}\varphi(y^{-1/\a}) y^{-1/\a-1}(1-\chi(y))dy.
\nonumber 
\end{eqnarray}
%for $L = \delta^\a, M=2\delta^\a$. 
Since it is easy to see that $E(t)$ 
is a $C^{\infty}$ function on $\R$, 
it suffices to consider the property of $J(t)$. 
%%%

The Taylor formula implies that
for any $N\in\N$, 
\begin{equation}\label{eqn:3.2}
\varphi(x)= \sum_{n=0}^{N-1} a_n x^n
	+ x^{N}\rho_N(x)
\end{equation}
where $a_n:=\varphi^{(n)}(0)/n!$ for $n\in\Z_+$ and 
$\rho_N(x)$ is a $C^{\infty}$ function on $\R$.

%%%

Substituting (\ref{eqn:3.2}) into (\ref{eqn:3.1}), 
we have
\begin{equation}\label{eqn:3.3}
J(t) = 
\sum_{n=0}^{N-1} 
\frac{a_n}{\a} F_{\gamma_n}(t)
%\int_r^{\infty}e^{ity}y^{-\frac{j+1}{\a}-1}\chi(y)dy
+ R_N(t),
\end{equation}
where 
$F_{\gamma_n}(t)$ is as in (\ref{eqn:2.2}) with 
$\gamma_n:=(n+1)/\alpha$ for $n\in\Z_+$ and 
$$	
R_N(t)
=\frac{1}{\a}\int_0^{\infty}
e^{ i  ty}y^{-\frac{N+1}{\a}-1}\rho_N(y^{-1/\a})\chi(y)dy.
$$	
%%%%%%%%%
Since the function 
$y \mapsto \rho_N(y^{-1/\alpha})\chi(y)$ is
bounded on $\R$, 
it is easy to see that $R_N(t)$ is a 
$C^{\lceil \frac{N+1}{\alpha}\rceil-1}$ function on $\R$. 
We will consider the first term in (\ref{eqn:3.3}).

\medskip 

(i)\quad 
First, let us consider the case where 
$\alpha$ is a rational number. 

By applying Lemmas~2.2 and 2.3,
we see that the first term in (\ref{eqn:3.3}) 
can be expressed as
\begin{equation}\label{eqn:3.4}
\sum_{n=0}^N 
\frac{a_n}{\a} F_{\gamma_n}(t)\equiv 
P_N(t)+Q_N(t) \quad \mbox{mod $C^{\infty}(\R_+)$,}
\end{equation}
with
\begin{equation*}\label{eqn:}
P_N(t)=	
\sum_{n\in {S}_N} 
\frac{a_n}{\a}\tilde{A}_{\gamma_n}t^{\gamma_n}, \quad
Q_N(t)=
\sum_{j\in T_N}
\frac{a_n}{\a} \tilde{B}_{\gamma_n}t^{\gamma_n} 
\log t, 
\end{equation*}
where 
$
S_N:=\{n\in\Z_+:\gamma_n\not\in\Z,\,\, n\leq N-1\}
$
and 
$
T_N:=\{n\in\Z_+:\gamma_n\in\Z,\,\, n\leq N-1\}.
$

Putting (\ref{eqn:3.3}),  (\ref{eqn:3.4}) 
and the above regularities of $R_N(t)$ together, 
we can easily show (i) in the theorem. 
%%%%%%%%%%%%%%%%%%%%%%%%%%%%%%%%%%%%%%%%%%%%

\medskip

(ii)\quad 
Since the case where $\alpha$ is not rational 
can be more easily dealt with 
and the equation (\ref{eqn:1.6}) in the theorem 
can be obtained in a similar fashion 
to the case of (i), 
the proof will be left to the readers.

%%%%%%%%%%%%%%%%%%%%%
%%%%%%%%%%%%%%%%%%%%%
%%%%%%%%%%%%%%%%%%%%%

\section{Generalized Fresnel integrals}

In this section,
we compute the exact values of the improper integrals
\begin{equation}\label{eqn:4.1}
\int_0^{\infty}e^{\pm  i  x^{\frac{1}{\alpha}}}dx
=\lim_{R\to\infty}
\int_0^{R}e^{\pm  i  x^{\frac{1}{\alpha}}}dx,
\end{equation}
where $0<\alpha<1$, which gives a proof of Lemma~2.1.
When $\alpha=1/2$, the above integrals can be explicitly 
computed as 
\begin{equation*}
\int_0^{\infty}e^{\pm  i  x^{2}}dx=
\frac{\sqrt{\pi}}{2\sqrt{2}}(1\pm i )=
\frac{\sqrt{\pi}}{2} e^{\pm\frac{\pi}{4} i }
\end{equation*}
by using an elementary method in complex analysis. 
These equations imply 
\begin{equation*}
\int_{0}^{\infty} \sin(x^2)dx=
\int_0^{\infty}\cos(x^2)dx=
\frac{\sqrt{\pi}}{2\sqrt{2}}. 
\end{equation*}  
The above two integrals are called 
the Fresnel integrals.  
%%%
The integral (\ref{eqn:4.1}) can also be explicitly computed 
in a similar fashion to the case of the Fresnel integrals
(see \cite{Nag20}, \cite{NaM20}, \cite{NaM22}, 
\cite{Tak21}, etc.).
Since this computation is essential to our analysis, 
we will give an exact proof of the following proposition. 

%%%%%%%%%%%%%%%%%%%%%%%%%%%%
\begin{proposition}
\begin{equation}\label{eqn:4.2}
\int_0^{\infty}e^{\pm  i  x^{\frac{1}{\a}}}dx
=\alpha\int_0^{\infty} e^{\pm  i  y} y^{\alpha-1}dy
=e^{\pm\frac{\alpha}{2}\pi i }
\Gamma(\alpha+1),
\end{equation}
where $0<\a<1$. 
\end{proposition}
%%%%%%%%%%%%%%%%%%%%%%%%%%%%%

\begin{proof}
We remark that 
the first equality can be seen
by exchanging the integral variable.   
In this proof, 
we deal with only the case of sign "$+$".

First, we prepare the four integral curves as follows.
\begin{equation*}
\begin{split}
&C_1: z=x \,\,\,(\varepsilon\leq x\leq R); \\
&C_2: z=Re^{ i \theta} \,\,\, (0\leq\theta\leq\pi/2); \\
&C_3: z= i y \,\,\, (\varepsilon\leq y\leq R);\\
&C_4: z=\varepsilon e^{i\theta} 
\,\,\, (0\leq\theta\leq\pi/2),
\end{split}
\end{equation*}
where $\varepsilon, R$ are positive numbers 
with $\varepsilon < R$ and each curve admits
a direction, which can be determined as in 
the figure below.
The anti-clockwise oriented closed curve
$\sum_{j=1}^4 C_j$ is denoted by $C$ and
the bounded domain surrounded by $C$ 
is denoted by $D$. 

%%%%%%%%%%%%%%%%%%%%%%%%%%%%%%%%%%%%%%%%%
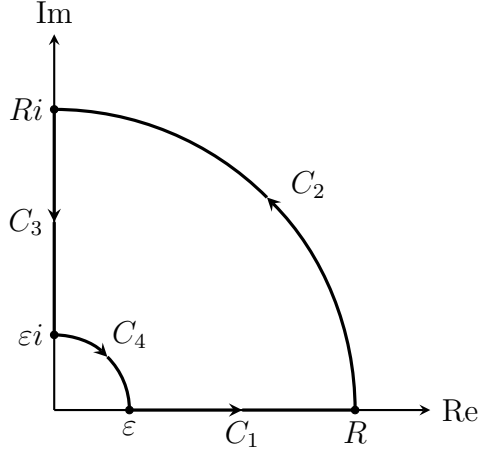
\begin{center}
\begin{tikzpicture}%[x=0.55cm,y=0.54cm]
  
\draw[thick,-stealth] (0,0)--(5,0) node [anchor=west]{$\mathrm{Re}$};
\draw[thick,-stealth] (0,0)--(0,5) node [anchor=south]{$\mathrm{Im}$};

\node at (1,0) [below] {$\varepsilon$};
\node at (4,0) [below] {$R$};
\node at (0,1) [left] {$\varepsilon  i $};
\node at (0,4) [left] {$R  i $};

\filldraw[fill=black] (1,0) circle[radius=0.5mm];
\filldraw[fill=black] (4,0) circle[radius=0.5mm];
\filldraw[fill=black] (0,1) circle[radius=0.5mm];
\filldraw[fill=black] (0,4) circle[radius=0.5mm];

\draw[very thick,-stealth] (0,1) arc(90:45:1);
\draw[very thick] ({sqrt(2)/2},{sqrt(2)/2}) arc(45:0:1);
\draw[very thick,-stealth] (4,0) arc(0:45:4);
\draw[very thick] ({2*sqrt(2)},{2*sqrt(2)}) arc(45:90:4);
\draw[very thick,-stealth] (1,0)--(2.5,0);
\draw[very thick] (2.5,0)--(4,0);
\draw[very thick,-stealth] (0,4)--(0,2.5);
\draw[very thick] (0,2.5)--(0,1);

\node at (2.5,0) [below] {$C_1$};
\node at (3,3) [right] {$C_2$};
\node at (0,2.5) [left] {$C_3$};
\node at (1,1) {$C_4$};
   
\end{tikzpicture} 
\end{center}
%%%%%%%%%%%%%%%%%%%%%%%%%%%%%%%%%%

Let $f(z)=e^{ i  z}z^{\a-1}$.
Since $f(z)$ is holomorphic near $\overline{D}$,
the Cauchy integral theorem implies
\begin{equation}\label{eqn:4.3}
	\int_C f(z) dz
	=\sum_{j=1}^4\int_{C_j} f(z) dz=0.
\end{equation}

The integral with respect to $C_2$ can be estimated as
\begin{equation}\label{eqn:4.4}
	\left| \int_{C_2} f(z) dz \right|
	\le R^{\a}\int_0^{\frac{\pi}{2}}e^{-R\sin \theta}d\theta
	\le R^{\a} \int_0^{\frac{\pi}{2}} e^{-\frac{2R}{\pi}\theta}d\theta
	=\frac{\pi}{2}R^{\a-1}(1-e^{-R})
\end{equation}
by using  Jordan's inequality. 
Note that the above last term tends to $0$ as $R\to \infty$.

In order to estimate the integral with respect to $C_4$, 
we define
$\psi(x)=(1-e^{-x})/x$ for $x\neq 0$ and  $=1$ for $x=0$, 
which is a $C^{\infty}$ function on $\R$.
By using $\psi$, we have
\begin{align}\label{eqn:4.5}
	\left| \int_{C_4} f(z)dz \right|
	\le \varepsilon^{\a}\int_0^{\frac{\pi}{2}} e^{-\frac{2\varepsilon}{\pi}\theta}d\theta
	= \frac{\pi}{2}\varepsilon^{\a-1}(1-e^{-\varepsilon})
	=\frac{\pi}{2}\varepsilon^{\a}\psi(\varepsilon).
\end{align}
Note that the above last term tends to $0$ as $\varepsilon \to 0$.

On the other hand, the integrals with respect to 
$C_1,C_3$ can be expressed as
\begin{equation}\label{eqn:4.6}
\begin{split}
&\int_{C_1} f(z) dz 
=\int_{\varepsilon}^{R}e^{ i x}x^{\a-1}dx;\\
&\int_{C_3} f(z) dz 
=-e^{\frac{\a\pi}{2} i }
\int_{\varepsilon}^{R}e^{-x}x^{\a-1}dx.
\end{split}
\end{equation}
Applying (\ref{eqn:4.4}), (\ref{eqn:4.5}), (\ref{eqn:4.6}) 
to (\ref{eqn:4.3}) and 
considering the limits $R \to \infty$, $\varepsilon \to 0$, 
we have
\begin{equation*}
\begin{split}
	\int_{0}^{\infty}
	e^{ i x}x^{\a-1}dx&=e^{\frac{\a\pi}{2} i }\int_0^{\infty}e^{-x}x^{\a-1}dx
	=e^{\frac{\a\pi}{2} i }\Gamma(\a)\\
	&=\frac{1}{\a}e^{\frac{\a\pi}{2} i }\Gamma(\a+1).
\end{split}
\end{equation*}
%The last equality is from the property of gamma function $\Gamma(n+1)=n\Gamma(n)$.

\end{proof}

\begin{remark}
The above method is not available in
the case where $\alpha\geq 1$.   
The interesting papers \cite{Nag20}, \cite{NaM20}, \cite{NaM22}  
deal with this general case and 
show that 
the equalities in Proposition~4.1 hold in some sense.   
\end{remark}

%%%%%%%%%%%%%%%%%%%%%%%%%%%%%%%%%%%%%%%%%%%%%%%%%%%%%%%%%%%%%%%%%%%%%%%%%%%%%%%%%%%%

\section{In the case of the Laplace integral}

Let us consider the Laplace integral
%%%
\begin{equation}\label{eqn:5.1}
	L(t)=
\int_{-\infty}^{\infty} e^{-t f(x)} \varphi(x)dx
	\quad\,\,\, \text{for $t \in \R$,}
\end{equation}
where $f$ and $\varphi$ are the same as those 
in (\ref{eqn:1.1}) and, moreover,
$f$ has a minimum at the origin.
If there is $k\in\N$ with $k\geq 2$ such that 
$f'(0)=\cdots=f^{(k-1)}(0)=0$ and $f^{(k)}(0)\neq 0$,  
then for any positive integer $N$, 
\begin{equation}\label{eqn:5.2}
L(t)=\sum_{n=1}^{N-1} \hat{c}_n t^{-n/k} + O(t^{-N/k})
\quad \mbox{as $t \to +\infty$},
\end{equation}
where $\hat{c}_n$ are constants depending 
on $f$ and $\varphi$, and the exact value of 
the first coefficient $\hat{c}_1$ can be explicitly given.
The result in (\ref{eqn:5.2}) 
can be easily shown in a similar fashion to 
that in the case of oscillatory integrals. 

In this section, we will consider 
the properties of integrals of the form
\begin{equation}\label{eqn:5.3}
L_{\alpha}(t)=
\int_{0}^{\infty}\exp
\left(-t\frac{1}{x^{\alpha}}
\right) \varphi(x)dx
\quad\,\,\, \text{for $t \in \R$,}
\end{equation}
where $\alpha$ is a positive real number and 
$\varphi$ is as in (\ref{eqn:5.1}). 
%%%
In the case of the above integral, 
we will give a result analogous to Theorem~1.1.

%%%%%%%%%%%%%%%%%%%%%%
\begin{theorem}\label{thm:5.1}
%Let $\alpha$ be a positive real number.
(i)\quad 
If $\alpha>0$ is a rational number, 
then for any positive integer $N$, %with $N>1/\alpha$, 
\begin{equation}\label{eqn:5.4}
L_{\alpha}(t) = 
\sum_{n=1}^{N} \hat{A}_{n} t^{n/\alpha}+
\sum_{n=1}^{N} \hat{B}_{n} t^{n/\alpha} \log t
+\hat{\psi}_N(t)
\quad \mbox{ for $t\in\R_+$,}
\end{equation}
where $\hat{\psi}_N(t)$ is a 
$C^{\lceil \frac{N+1}{\alpha}\rceil-1}$ function on $\R_+$ and   
\begin{equation}\label{eqn:5.5}
\begin{split}
&\hat{A}_n=\frac{1}{\alpha}
\frac{\varphi^{(n-1)}(0)}{(n-1)!}
\Gamma(-n/\alpha) \,\,\,
\mbox{if $n/\alpha \not\in\N$}, \quad 
\hat{A}_n=0 \,\,\, \mbox{if $n/\alpha\in\N$,}\\
&\hat{B}_n=\frac{-1}{\alpha}
\frac{\varphi^{(n-1)}(0)}{(n-1)!}
\frac{1}{(n/\alpha)!}
   \,\,\, 
\mbox{if $n/\alpha \in\N$},   \quad 
\hat{B}_n=0 \,\,\, \mbox{if $n/\alpha\not\in\N$,}
\end{split}
\end{equation}
for $n\in\N$.

(ii)\quad 
If $\alpha>0$ is not a rational number,  then
for any positive integer $N$, %with $N>1/\alpha$, 
\begin{equation}\label{eqn:5.6}
L_{\alpha}(t) = 
\sum_{n=1}^{N} \hat{A}_{n} t^{n/\a}+\hat{\phi}_N(t)
\quad \mbox{ for $t\in\R_+$,}
\end{equation}
where $\hat{\phi}_N(t)$ is a $C^{\lceil \frac{N+1}{\alpha}\rceil-1}$ 
function on $\R$ and 
$\hat{A}_n$ are as in (\ref{eqn:5.5}).
\end{theorem}
%%%%%%%%%%%%%%%%%%%%%%%%%%%%

Since the proof of the above theorem 
can be easily given by the same way 
as that of Theorem~1.1, 
it is omitted.

\bigskip

{\it Acknowledgements.} \quad 
We express our sincere thanks to  
Sheehan Olver and Ory Schnitzer for kindly informing us of 
the method in the book \cite{Hin91}, which plays a crucial role
in the proof of Lemma~2.3. 
We also thank Shuhuang Xiang   
for explaining computation by using
the incomplete Gamma function 
and  
Tatsuki Takeuchi for discussion about 
generalized Fresnel integrals. 
The referee kindly gave us many useful comments. 
This work was supported by 
JSPS KAKENHI Grant Numbers JP20K03656, JP20H00116.

\end{document}